\documentclass{aims}
\usepackage{amsmath}
  \usepackage{paralist}
  \usepackage{graphics} 
  \usepackage{epsfig} 
\usepackage{graphicx} 
 \usepackage[colorlinks=true]{hyperref}
 
 \usepackage{pstricks, pst-node, pst-plot, pst-circ}
 \usepackage{moredefs}
\usepackage{tikz}
\hypersetup{urlcolor=blue, citecolor=red}

\newrgbcolor{dblue}{0.2 0.3 0.38}
\newrgbcolor{fg}{0.2 0.3  0.38}
\newrgbcolor{bg}{0.92 0.93 0.95}

\def\currentvolume{}
 \def\currentissue{}
  \def\currentyear{}
   \def\currentmonth{}
    \def\ppages{}
     \def\DOI{}
\newtheorem{theorem}{Theorem}[section]
\newtheorem{corollary}{Corollary}

\theoremstyle{definition}

\newtheorem{remark}{Remark}

\def\cR{\mathbb{R}}
\def\cR{\mathbb{R}}
\def\dis{\displaystyle}

\def\L{\ell}
\def\e{\mbox{e}}

\def\P{{\mathcal P}}

\title[Second--Order Herglotz Problem on Spheres] 
      {Variational and Optimal Control Approaches for the Second--Order Herglotz Problem on Spheres$^*$}

\author[L. Machado,  L. Abrunheiro and N. Martins]{}


 \email{abrunheiroligia@ua.pt}
 \email{lmiguel@utad.pt}
 \email{natalia@ua.pt}

\thanks{$^{\dagger}$Corresponding author: lmiguel@utad.pt}
\thanks{$^{*}$This is a preprint of a paper whose final and definite form will appear in the \textit{Journal of Optimization Theory and Applications}. Paper submitted 25-Nov-2017; accepted for publication 19-Oct-2018.}

\begin{document}
\maketitle


\centerline{\scshape Lu\'{i}s Machado$^\dagger$}
\medskip
{\footnotesize
 \centerline{ISR $-$ University of Coimbra, Coimbra, Portugal \&}
 \centerline{Department of Mathematics, University of Tr\'as-os-Montes e Alto Douro (UTAD),}
   \centerline{Vila Real, Portugal}
}

\medskip
\centerline{\scshape L\'{i}gia Abrunheiro}
\medskip
{\footnotesize
 \centerline{Center for Research and Development in Mathematics and Applications (CIDMA) \& }
  \centerline{Higher Institute of Accounting and Administration,}
   \centerline{University of Aveiro, Portugal}
} 

\medskip

\centerline{\scshape Nat\'{a}lia Martins}
\medskip
{\footnotesize
 \centerline{Center for Research and Development in Mathematics and Applications (CIDMA) \&}
   \centerline{Department of Mathematics, University of Aveiro, Portugal}
}

\bigskip


\begin{abstract}
The present paper extends the classical second--order variational problem of Herglotz type to the more general context of the Euclidean sphere $S^n$ following variational and optimal control approaches. The relation between the Hamiltonian equations and the generalized Euler-Lagrange equations is established. This problem covers some classical variational problems posed on the Riemannian manifold $S^n$ such as the problem of finding cubic polynomials on $S^n$. It also finds applicability on the dynamics of the simple pendulum in a resistive medium.
\end{abstract}

\section{Introduction}

It is well known that for a very broad class of nonconservative field phenomena, we are not able to put them into the classical framework of Hamilton's variational principle. One may think, for instance, in the partial differential equation that describes the conduction of heat in solids \cite{VuJo}. To fill this gap, in 1930, while working on contact transformations and its connections to Hamiltonian systems and Poisson brackets, Gustav Herglotz \cite{Herglotz1930} generalized the classical variational problem. This problem consists in finding a curve in the configuration space and a scalar function sought by a differential equation whose final value is extremized. The Herglotz problem differs from the classical calculus of variations problem since the differential equation depends, not only on time, the curve and their derivatives, but also on the scalar function itself.

Besides its pioneer importance in thermodynamics via contact transformations, the generalized variational principle of Herglotz had also a profound impact in applications from physics, engineering and other mathematically related sciences. We reinforce the important fact that it provides a variational description of nonconservative and dissipative processes even when the Lagrangian is autonomous \cite{Georgieva2002}. 

The Herglotz problem has been a source of inspiration for several authors after the two publications \cite{MR1391230,Guenther1996} in the late nineties till nowadays. The most well known results from the classical calculus of variations have been generalized for the variational problem of Herglotz. We highlight the two Noether theorems for the first--order problem \cite{Georgieva2002,Georgieva2010,Georgieva2005,Georgieva2003} that were generalized for the Herglotz problem with time delay \cite{Simao+Delfim+Natalia2015b} and for the higher--order Herglotz problem, with and without, time delay \cite{Simao+Delfim+Natalia2015a,Simao+Delfim+Natalia2015c,Simao+Delfim+Natalia2016,Simao+Delfim+Natalia2018}.  All the aforementioned results have been derived following Lagrangian or Hamiltonian formalisms, but recently, Noether's Theorem and its inverse have been also extended for variational problems of Herglotz type for Birkhoffian systems  \cite{Zhang,Tian}.  The theory of Birkhoffian mechanics is a natural generalization of the Hamiltonian mechanics and is based on the Pfaff-Birkhoff principle and Birkhoff's equations. Thus, Birkhoffian mechanics can be applied not only to Hamiltonian, Lagrangian, or Newtonian mechanics, but also to quantum mechanics, statistical mechanics, holonomic and nonholonomic mechanics, atomic and molecular physics among others (see \cite{Birk,San} and references therein for more details).

We also mention the fractional variational Herglotz problem and its corresponding Noether's Theorem discussed in \cite{Alm:2014}.  Fractional calculus deals with derivatives and integrals of arbitrary order, thus extending the capabilities of classical calculus, but also introducing novelties in theoretical and applied research. In particular, fractional variational calculus has gained considerable popularity in the last decades due to its important applications in physics, classical and quantum mechanics, electrodynamics, field theory, cosmology and nanoscience \cite{Riewe1,Riewe2}. Nevertheless, recent investigations not only in engineering, but also in other fields, have shown that the dynamics of many systems are described more accurately using fractional differential equations. The increase demanding for efficiency, accuracy and high precision of systems required the development of the new field of fractional optimal control theory \cite{Agrawal1,Agrawal2,Jarad:10,Jarad:12,Bahaa}. 

The generalization of the Herglotz problem in the more general context of Riemannian manifolds is a challenging topic of research due to the complexity of dealing with covariant derivatives. The first attempt in this direction was the study of the first--order Herglotz problem on the Euclidean unit sphere $S^n$ \cite{2016AbrMacMar}.

Motivated by the significance of Herglotz variational problem and the important role played by Riemannian manifolds in many engineering and physics applications, in this paper we extend the second--order Herglotz problem to $S^n$ following two different formulations. 
The interest of studying this type of second--order variational problems relies on the fact that most part of mechanical systems are typically represented by second--order Newtonian systems \cite{Arnold}.  In the first formulation, a variational problem of Herglotz is posed by considering the Riemannian manifold $S^n$ as the configuration space and where the differential equation that characterizes the Herglotz problem depends on the second covariant derivative. The second formulation is determined by an optimal control problem with a system of controlled state space equations equivalent to the differential equation from the variational approach. 

There are several works establishing the equivalence between optimal control problems and the classical calculus of variations problem with holonomic or nonholonomic constraints (see, for instance, \cite{Cro:94,Cro:95,Cro:96}).  However, the variational problem of Herglotz does not fit into the classical framework and therefore those results cannot be applied directly to this specific problem. We believe that the trick to establish this equivalence for the Herglotz problem passes through the generalization of the well known Legendre-Ostrogradsky condition. However, this is still an open question.

The structure of the paper is as follows. In Section \ref{secsph}, we recall the most important results about the geometry of the Euclidean sphere $S^n$. Our main contributions appear in Sections \ref{2stOrderVP} and \ref{OControlP}. In Section \ref{2stOrderVP},  the second--order variational problem of Herglotz is formulated on $S^n$ and the corresponding generalized  \mbox{Euler--Lagrange} equations are derived and written in terms of the higher--order covariant derivatives.  In Section \ref{OControlP}, the problem is formulated using the framework of optimal control theory and the set of Hamiltonian equations is derived. The Euler--Lagrange equations are therefore obtained from the Hamiltonian equations. Finally, we illustrate our results with some examples. In particular, we fit the Riemannian cubic polynomials in this new context and show how the dynamical equation for the pendulum in a resistive medium can be obtained from our approach.


\section{Preliminaries on the Geometry of $S^n$} \label{secsph}

Let $\langle \cdot,\cdot \rangle$ denote the standard inner product in the Euclidean space $\cR^{n+1}$. The unit $n-$sphere $S^n:=\bigl\{p\in\cR^{n+1}\::\:\langle p, p \rangle =1\bigr\}$ is an $n-$dimensional Riemannian manifold embedding in $\cR^{n+1}$. The tangent space of $S^n$ at a point $p\in S^n$ and its orthogonal complement are given, respectively, by
$$
T_pS^n=\bigl\{v\in \cR^{n+1}\::\:\langle v, p\rangle =0\bigr\}~~\mbox{and}
~~T_p^\perp S^n=\bigl\{\alpha\, p\in\cR^{n+1}\::\:\alpha \in \cR\bigr\}.
$$

Clearly, any vector $u\in\cR^{n+1}$ can be uniquely written as
$$
u=u-\langle u,p\rangle p +\langle u,p\rangle p,
$$
where $u-\langle u,p\rangle\, p\in T_p S^n$ and $\langle u,p\rangle\, p\in T_p^\perp S^n$.

The tangent bundle of $S^n$ is the disjoint union of all tangent spaces of $S^n$ 
$$
TS^n=\bigcup_{p\in S^n}T_pS^n,
$$
and is a differentiable manifold of dimension $2n$. Each element $v\in T_pS^n$ is naturally identified with a pair $(p,v)$, for each $p\in S^n$. Thus, an element of $TS^n$ can be seen as a pair $(p,v)$, with $p\in S^n$ and $v\in T_pS^n$.  The cotangent bundle $T^*S^n$ is the dual bundle of $TS^n$, so each element of $T^*S^n$ can be seen as a pair $(p,\alpha)$, with $p\in S^n$ and $\alpha  \in T^*_pS^n$, $\alpha:  T_pS^n\to \mathbb{R}$. In a similar way, we define the $4n$-dimensional manifold $TTS^n$, $T TS^n =  \cup_{(p,v)}T_{(p,v)}TS^n$. 

Recall that the tangent bundle of a manifold is defined as the set of all equivalence classes of curves on the manifold that agree up to their derivative (roughly speaking). Extending this concept, we reach the geometry for higher--order tangent bundles. In the scope of this paper, we are interested in the second--order tangent bundle, $T^2S^n$, which is a $3n$-dimensional manifold that can be injected in $TTS^n$ and identified with 
$$
T^2S^n=\{(p,v, w,u)\in TTS^n: w=v \}.
$$
Moreover, each element of $T^2S^n$ can be identified with a triplet $(p,v,u)$, with $p\in S^n$ and $v,u\in T_pS^n$.

The covariant derivative of a smooth vector field $Y$ along a curve $x$ in $S^n$ is obtained by projecting, at time $t$, the usual derivative of $Y$, $\dot{Y}$, orthogonally onto $T_{x(t)}S^n$. Hence,
\begin{equation} \label{DCovaESFERA}
\frac{DY}{dt}(t)=\dot{Y}(t)-\bigl\langle \dot{Y}(t), x(t)\bigr\rangle\, x(t),
\end{equation}
and, more generally, 
\begin{equation*} \label{DCovaESFERA}
\frac{D^kY}{dt^k}=\dfrac{D}{dt}\Bigl(\dfrac{D^{k-1}Y}{dt^{k-1}}\Bigr), ~k\ge 2.
\end{equation*}

Since in our approach we are going to deal with second--order covariant derivatives, it can be easily seen from the above definitions that
\begin{equation*} \label{DCovaESFERA}
\frac{D^2Y}{dt^2}=\ddot{Y}-\bigl\langle \ddot{Y},x\bigr\rangle x -\bigl\langle \dot{Y},x\bigr\rangle \dot{x} .
\end{equation*}

For the particular case when $Y$ is the velocity vector field $\dot{x}$, its covariant derivative,  $\frac{D\dot{x}}{dt}$, called covariant acceleration and denoted by $\frac{D^2 x}{dt^2}$,  is simply given by
\begin{equation}  \label{cov_sph}
\frac{D^2x}{dt^2}= \ddot{x}-\bigl\langle\ddot{x},x\bigr\rangle x.
\end{equation}

See \cite{cmA,Jost} for further details about classical differential geometry and \cite{LeonRodrigues:1985} for details on the geometric properties of the higher--order tangent bundles.
\section{The Second--Order Herglotz Variational Problem on $S^n$} \label{2stOrderVP}

In this section, we present the second--order Herglotz problem on the Euclidean sphere $S^n$ and derive necessary optimality conditions for the existence of extremals; the so-called generalized Euler--Lagrange equations. This problem has a more complex structure than the first--order Herglotz problem studied in  \cite{2016AbrMacMar}, since it involves the covariant acceleration (\ref{cov_sph}), which is a nonlinear function. The problem can be formulated as follows.

\begin{quote}
{\textbf{Problem ($\P_V$)}:} 
Determine the trajectories $x\in C^4([0,T],S^n)$ \linebreak and $z\in C^1([0,T],\cR)$ that minimize the final value of the function $z$:
\begin{equation*}
\dis{\min_{(x,z)}}~~z(T), 
\end{equation*}
where $z$ satisfies the differential equation
\begin{equation}
\dot{z}(t)=L\Bigl(t,x(t),\dot{x}(t), \frac{D^2 x}{dt^2}(t),z(t)\Bigr), \:\:t\in[0,T], \label{diffequation}
\end{equation}
subject to the initial condition  
\begin{equation} 
z(0)=z_0 \label{init_condition}
\end{equation} 
and where $x$ satisfies the boundary conditions
\begin{equation} 
x(0)=x_0, \quad x(T)=x_T, \quad \dot{x}(0)=v_0 \quad \mbox{and} \quad \dot{x}(T)=v_T, \label{bound_condition}
\end{equation}
for some $x_0,x_T\in S^n$, $v_0\in T_{x_0}S^n$, $v_T\in T_{x_T}S^n$ and $z_0,\,T\in\cR$.
\end{quote}

The Lagrangian $L$ is assumed to satisfy the following hypotheses:
\begin{enumerate}
\item $L\in C^1([0,T]\times T^2S^n\times \cR, \mathbb{R})$;
\item Functions  
\begin{itemize}
\item[] $\displaystyle t \mapsto \frac{\partial L}{\partial x}\left(t, x(t), \dot{x}(t),\omega(t), z(t)\right)$, 

\item[] $\displaystyle t \mapsto \frac{\partial L}{\partial \dot{x}}\left(t, x(t), \dot{x}(t), \omega(t),z(t)\right)$, 

\item[] $\displaystyle t \mapsto \frac{\partial L}{\partial \omega}\left(t, x(t), \dot{x}(t), \omega(t),z(t)\right)$,
\item[] $\displaystyle t \mapsto \frac{\partial L}{\partial z}\left(t, x(t), \dot{x}(t), \omega(t),z(t)\right)$ 
\end{itemize}
are differentiable  for any admissible trajectory $(x,z)$. 
\end{enumerate}

In above, we considered $\omega = \frac{D^2 x}{dt^2}$ and denoted by $\frac{\partial L}{\partial x}$, $\frac{\partial L}{\partial \dot{x}}$ and $\frac{\partial L}{\partial \omega}$ the  functional  partial derivatives of the function $L$ with respect to $x$, $\dot{x}$ and $\omega$, respectively, that are elements of $T^*S^n$.



Following the approach given in \cite{2016AbrMacMar}, we will look to the above problem as a constrained optimization problem on the embedding Euclidean space $\cR^{n+1}$, where the trajectory $x$ is seen as a curve in $\cR^{n+1}$ satisfying the holonomic constraint
\begin{equation} \label{eq:constraint}
\bigl\langle x(t),x(t)\bigr\rangle=1,~~\forall t\in[0,T].
\end{equation}
In this case, $L\in C^1([0,T]\times\cR^{3n+3}\times\mathbb{R}, \mathbb{R})$.

An admissible variation of a solution $x$ of the constrained optimization problem can be defined by $x_{\varepsilon}:=x+\varepsilon h$, where  $\varepsilon$ is a real parameter and $h\in C^4([0,T],\cR^{n+1})$ is such that 
\begin{equation} \label{bound_cond}
h(0)=h(T)=
\dot{h}(0)=\dot{h}(T)=0.
\end{equation}
Assume also that $\ddot{h}(0)=0$.

For the sake of simplicity we sometimes suppress the arguments of the functions.

\begin{theorem} \label{T:ELsecondOrder}
If $(x,z)$ is a solution of problem {\rm ($\P_V$)}, then it satisfies the generalized Euler--Lagrange equation
\begin{align}
\begin{split}\label{E-L(2)}
& \dfrac{D^2}{dt^2}\Bigl(\dfrac{\partial L}{\partial \omega}\Bigr)-\dfrac{D}{dt}\Bigl(\dfrac{\partial L}{\partial \dot{x}}\Bigr) -\Bigl\langle \dfrac{\partial L}{\partial x},x\Bigr\rangle x-\bigl\langle \ddot{x},x\bigr\rangle \dfrac{\partial L}{\partial \omega}-2 \Bigl\langle \frac{\partial L}{\partial \omega}, 
\dot{x} \Bigr\rangle \dot{x}- 2 \Bigl\langle \frac{\partial L}{\partial \omega}, 
x \Bigr\rangle \ddot{x}  \\
& +\dfrac{\partial L}{\partial x}-\Bigl\langle \dfrac{d}{dt}\Bigl(\dfrac{\partial L}{\partial \omega}\Bigr),x\Bigr\rangle \dot{x}+3\Bigl\langle  \dfrac{\partial L}{\partial \omega},x\Bigr\rangle\bigl\langle\ddot{x},x\bigr\rangle x +\Bigl(\dfrac{\partial L}{\partial z}\Bigr)^2 \Bigl( \dfrac{\partial L}{\partial \omega}-\Bigl\langle \dfrac{\partial L}{\partial \omega},x\Bigr\rangle x\Bigr)
\\
&-\dfrac{D}{dt}\Bigl(\dfrac{\partial L}{\partial z} \dfrac{\partial L}{\partial \omega}\Bigr) -\dfrac{\partial L}{\partial z} \Bigl[ \dfrac{D}{dt}\Bigl(\dfrac{\partial L}{\partial \omega}\Bigr)- \dfrac{\partial L}{\partial \dot{x}} +\Bigl\langle \dfrac{\partial L}{\partial \dot{x}} ,x\Bigr\rangle x 
- 2 \Bigl\langle\dfrac{\partial L}{\partial \omega},x\Bigr\rangle \dot{x}\Bigr] =0.
\end{split}
\end{align}
\end{theorem}
\begin{proof}
In order to find first order necessary conditions for the constrained optimization problem, let us consider the functional defined by
\begin{equation*}
J\bigl(t,x,\dot{x},\omega,z,\lambda, \mu\bigr)=z(T)+\lambda\Bigl(\dot{z}-L\bigl(t,x,\dot{x},\omega,z\bigr)\Bigr)+\mu \bigl(\bigl\langle x,x\bigr\rangle-1\bigr),
\end{equation*}
where the scalar functions $\lambda$ and $\mu$ are Lagrange multipliers.
Therefore, a necessary condition for $x$ to be a solution for the proposed optimization problem is
\begin{equation} \label{first}
\left.\dfrac{d}{d\varepsilon}J\bigl(t,x_{\varepsilon},\dot{x}_{\varepsilon},\ddot{x}_{\varepsilon}- \bigl\langle \ddot{x}_{\varepsilon}, x_{\varepsilon}\bigr\rangle x_{\varepsilon},z,\lambda, \mu\bigr)\right|_{\varepsilon=0}=0, 
\end{equation}
for all admissible admissible variations $x_{\varepsilon}$ of  $x$. Let $\xi$  be the first variation of $z$:
\begin{equation*}
\xi(t)=\left.\dfrac{d}{d\varepsilon}z\bigl(t,x_{\varepsilon},\dot{x}_{\varepsilon},\ddot{x}_{\varepsilon}- \bigl\langle \ddot{x}_{\varepsilon}, x_{\varepsilon}\bigr\rangle x_{\varepsilon}\bigr)\right|_{\varepsilon=0}, 
\end{equation*}
where for simplicity of notation we suppressed the dependence of $t$ on $x_{\varepsilon}$.

Notice that $h(0)=\dot{h}(0)=\ddot{h}(0)=0$ and so  $\xi(0)=0$. In order to conclude that $\xi(T)=0$, define the real valued function $g$ by
$$
g(\varepsilon)= z\bigl(T,x_{\varepsilon},\dot{x}_{\varepsilon},\ddot{x}_{\varepsilon}- \bigl\langle \ddot{x}_{\varepsilon}, x_{\varepsilon}\bigr\rangle x_{\varepsilon}\bigr).
$$
Since, by hypothesis, $T$ is a minimizer of $z$, one concludes that $0$ is a minimizer of $g$.  So,  the differentiability of the function $g$ implies that
\begin{flalign*}0=g'(0)=\left.\dfrac{d}{d\varepsilon}z\bigl(T,x_{\varepsilon},\dot{x}_{\varepsilon},\ddot{x}_{\varepsilon}- \bigl\langle \ddot{x}_{\varepsilon}, x_{\varepsilon}\bigr\rangle x_{\varepsilon}\bigr)\right|_{\varepsilon=0}=\xi(T).
\end{flalign*} 

Moreover, we have
\begin{flalign*}
\dot{\xi}(t)&=
\left.\dfrac{d}{d\varepsilon}L\bigl(t,x_{\varepsilon},\dot{x}_{\varepsilon},\ddot{x}_{\varepsilon}- \bigl\langle \ddot{x}_{\varepsilon}, x_{\varepsilon}\bigr\rangle x_{\varepsilon} ,z\bigr)\right|_{\varepsilon=0} . 
\end{flalign*}

Therefore,
\begin{align*}
& \left.\dfrac{d}{d\varepsilon}J\bigl(t,x_{\varepsilon},\dot{x}_{\varepsilon},\ddot{x}_{\varepsilon}- \bigl\langle \ddot{x}_{\varepsilon}, x_{\varepsilon}\bigr\rangle x_{\varepsilon},z,\lambda, \mu\bigr)\right|_{\varepsilon=0} = \lambda \Bigl( \dfrac{d \xi}{d t} - \Bigl\langle \dfrac{\partial L}{\partial x},h \Bigr\rangle -  \Bigl\langle \dfrac{\partial L}{\partial \dot{x}},\dot{h}\Bigr\rangle \Bigr)\\
&\quad -  \lambda \dfrac{\partial L}{\partial z} \xi - \lambda\Bigl\langle \dfrac{\partial L}{\partial \omega},\ddot{h} -\langle \ddot{h},x \rangle x -\langle \ddot{x},h\rangle x-\langle \ddot{x},x\rangle h\Bigr\rangle +2\mu \bigl\langle x,h\bigr\rangle.
\end{align*}
Now condition (\ref{first}) is equivalent to 
\begin{align*}
 \dfrac{d \xi}{d t} - \dfrac{\partial L}{\partial z} \xi &= \Bigl\langle \dfrac{\partial L}{\partial x} -\Bigl\langle \dfrac{\partial L}{\partial \omega},x\Bigr\rangle\ddot{x} - \langle \ddot{x},x\rangle \dfrac{\partial L}{\partial \omega} -2\frac{\mu}{\lambda} x ,h \Bigr\rangle + \Bigl\langle \dfrac{\partial L}{\partial \dot{x}},\dot{h}\Bigr\rangle \\
& \quad +\Bigl\langle \dfrac{\partial L}{\partial \omega} -\Bigl\langle \frac{\partial L}{\partial \omega}, x \Bigr\rangle x,\ddot{h}\Bigr\rangle
\end{align*}
which is a first order linear differential equation. Multiply both members of the above equation by 
$
I(t)=\e^{-\int_0^t \frac{\partial L}{\partial z}\,d\tau},
$
to get
\begin{align*}
\quad \dfrac{d}{dt}\bigl( I(t) \xi (t) \bigr) &=  I(t) \Bigl[\Bigl\langle \dfrac{\partial L}{\partial x} -\Bigl\langle \dfrac{\partial L}{\partial \omega},x\Bigr\rangle\ddot{x} - \langle \ddot{x},x\rangle \dfrac{\partial L}{\partial \omega} -2\frac{\mu}{\lambda} x ,h \Bigr\rangle + \Bigl\langle \dfrac{\partial L}{\partial \dot{x}},\dot{h}\Bigr\rangle\Bigr.\\
& \quad\qquad \Bigl.+\Bigl\langle \dfrac{\partial L}{\partial \omega} -\Bigl\langle \frac{\partial L}{\partial \omega}, x \Bigr\rangle x,\ddot{h}\Bigr\rangle\Bigr] .
 \end{align*}
Now, integrating both sides of the previous equation and using the fact that $\xi(0)=0$ and the boundary conditions (\ref{bound_cond}), it yields
\begin{align*}
\begin{split}
&\quad ~I(t) \xi (t)\\
&= \int_0^t I(\tau) \Bigl\langle  \dfrac{\partial L}{\partial x}-\dfrac{d}{d\tau}\Bigl(\dfrac{\partial L}{\partial \dot{x}}\Bigr)-\Bigl\langle \dfrac{\partial L}{\partial \omega},x\Bigr\rangle\ddot{x}- \langle \ddot{x},x\rangle \dfrac{\partial L}{\partial \omega} + \dfrac{\partial L}{\partial z}\dfrac{\partial L}{\partial \dot{x}}  -2\frac{\mu}{\lambda} x,h \Bigr\rangle\, d \tau 
\\
& \quad -\Bigl. I(\tau)\Bigl\langle \dfrac{d}{d\tau}\Bigl(\dfrac{\partial L}{\partial \omega} -\Bigl\langle \frac{\partial L}{\partial \omega}, x \Bigr\rangle x\Bigr) - \dfrac{\partial L}{\partial z}\Bigl(\dfrac{\partial L}{\partial \omega} - \Bigl\langle \dfrac{\partial L}{\partial \omega},x\Bigr\rangle x\Bigr)-\left.\dfrac{\partial L}{\partial \dot{x}},h \Bigr\rangle\right|_0^{t}
\\
&\quad + \Bigl. I(\tau) \Bigl\langle \dfrac{\partial L}{\partial \omega} -\left.\Bigl\langle \frac{\partial L}{\partial \omega}, x \Bigr\rangle x ,\dot{h} \Bigr\rangle\right|_{0}^t 
\\
&\quad +\int_0^t I(\tau) 
\Bigl\langle \dfrac{d^2}{d\tau^2}\Bigl(\dfrac{\partial L}{\partial \omega}-\Bigl\langle \frac{\partial L}{\partial \omega}, x \Bigr\rangle x \Bigr) -\dfrac{d}{d\tau}\Bigl(\dfrac{\partial L}{\partial z} \dfrac{\partial L}{\partial \omega}\Bigr) +\Bigl\langle \dfrac{d}{d\tau}\Bigl(\dfrac{\partial L}{\partial z}\dfrac{\partial L}{\partial \omega}\Bigr),x\Bigr\rangle x 
\\
& 
\quad +\dfrac{\partial L}{\partial z} \Bigl\langle\dfrac{\partial L}{\partial \omega},\dot{x}\Bigr\rangle x +\dfrac{\partial L}{\partial z}\Bigl\langle  \dfrac{\partial L}{\partial \omega},x\Bigr\rangle \dot{x} -\dfrac{\partial L}{\partial z}  \dfrac{d}{d\tau}\Bigl(\dfrac{\partial L}{\partial \omega} -\Bigl\langle \frac{\partial L}{\partial \omega}, x \Bigr\rangle x\Bigr)
\\
& \quad +\Bigl(\dfrac{\partial L}{\partial z}\Bigr)^2\Bigl( \dfrac{\partial L}{\partial \omega}-\Bigl\langle  \dfrac{\partial L}{\partial \omega},x\Bigr\rangle x\Bigr),h \Bigr\rangle d \tau.
\end{split}
\end{align*}
Evaluate the above for $t=T$, take into account the boundary conditions (\ref{bound_cond}), the fact that $\xi(T)=0$ and write
\begin{flalign*}
& \int_0^T I(t) 
\Bigl\langle \dfrac{d^2}{d\tau^2}\Bigl(\dfrac{\partial L}{\partial \omega}-\Bigl\langle \frac{\partial L}{\partial \omega}, x \Bigr\rangle x \Bigr) -\dfrac{d}{d\tau}\Bigl(\dfrac{\partial L}{\partial z} \dfrac{\partial L}{\partial \omega}\Bigr) +\Bigl\langle \dfrac{d}{d\tau}\Bigl(\dfrac{\partial L}{\partial z}\dfrac{\partial L}{\partial \omega}\Bigr),x\Bigr\rangle x \\
&
\quad +\dfrac{\partial L}{\partial z} \Bigl\langle\dfrac{\partial L}{\partial \omega},\dot{x}\Bigr\rangle x +\dfrac{\partial L}{\partial z} \Bigl\langle \dfrac{\partial L}{\partial \omega},x\Bigr\rangle \dot{x} -\dfrac{\partial L}{\partial z}  \dfrac{d}{d\tau}\Bigl(\dfrac{\partial L}{\partial \omega} -\Bigl\langle \frac{\partial L}{\partial \omega}, x \Bigr\rangle x\Bigr)\\
&\quad +\Bigl(\dfrac{\partial L}{\partial z}\Bigr)^2 \Bigl(\dfrac{\partial L}{\partial \omega}-\Bigl\langle \dfrac{\partial L}{\partial \omega},x\Bigr\rangle x\Bigr) + \dfrac{\partial L}{\partial x}-\Bigl\langle \dfrac{\partial L}{\partial \omega},x\Bigr\rangle\ddot{x}- \langle \ddot{x},x\rangle \dfrac{\partial L}{\partial \omega}\\
& \quad -\dfrac{d}{d\tau}\Bigl(\dfrac{\partial L}{\partial \dot{x}}\Bigr) + 
 \dfrac{\partial L}{\partial z}\dfrac{\partial L}{\partial \dot{x}}  -2\frac{\mu}{\lambda} x  ,h \Bigr\rangle d \tau=0.
\end{flalign*} 
Since the above condition holds for all admissible curves $h$ and  $I$ is a positive real valued function, one must have
\begin{flalign}
\begin{split}
&
\dfrac{d^2}{dt^2}\Bigl(\dfrac{\partial L}{\partial \omega}-\Bigl\langle \frac{\partial L}{\partial \omega}, x \Bigr\rangle x \Bigr)-\dfrac{d}{dt}\Bigl(\dfrac{\partial L}{\partial \dot{x}}\Bigr) + \dfrac{\partial L}{\partial x}-\langle \ddot{x},x\rangle \dfrac{\partial L}{\partial \omega} -\Bigl\langle \dfrac{\partial L}{\partial \omega},x\Bigr\rangle\ddot{x}
\\
& -\dfrac{\partial L}{\partial z}  \dfrac{d}{dt}\Bigl(\dfrac{\partial L}{\partial \omega} -\Bigl\langle \frac{\partial L}{\partial \omega}, x \Bigr\rangle x\Bigr)  -\dfrac{d}{dt}\Bigl(\dfrac{\partial L}{\partial z} \dfrac{\partial L}{\partial \omega}\Bigr) +\Bigl\langle \dfrac{d}{dt}\Bigl(\dfrac{\partial L}{\partial z}\dfrac{\partial L}{\partial \omega}\Bigr),x\Bigr\rangle x 
+\dfrac{\partial L}{\partial z}\dfrac{\partial L}{\partial \dot{x}}\\
& 
+\dfrac{\partial L}{\partial z} \Bigl\langle\dfrac{\partial L}{\partial \omega},\dot{x}\Bigr\rangle x +\dfrac{\partial L}{\partial z} \Bigl\langle \dfrac{\partial L}{\partial \omega},x\Bigr\rangle \dot{x}  +\Bigl(\dfrac{\partial L}{\partial z}\Bigr)^2 \Bigl(\dfrac{\partial L}{\partial \omega}-\Bigl\langle  \dfrac{\partial L}{\partial \omega},x\Bigr\rangle x\Bigr) -2\frac{\mu}{\lambda} x =0.
\end{split} \label{equacao}
\end{flalign} 
In order to proceed, note that
\begin{flalign*}
&
\quad ~\dfrac{d}{dt}\Bigl(\dfrac{\partial L}{\partial \omega}-\Bigl\langle \frac{\partial L}{\partial \omega}, x \Bigr\rangle x \Bigr) =  \dfrac{D}{dt}\Bigl(\dfrac{\partial L}{\partial \omega}\Bigr)-\Bigl\langle \frac{\partial L}{\partial \omega}, \dot{x} \Bigr\rangle x-\Bigl\langle \frac{\partial L}{\partial \omega}, x \Bigr\rangle \dot{x} 
\end{flalign*}
and
\begin{flalign*}
 \dfrac{d^2}{dt^2}\Bigl(\dfrac{\partial L}{\partial \omega}-\Bigl\langle \frac{\partial L}{\partial \omega}, x \Bigr\rangle x \Bigr) &= \dfrac{D^2}{dt^2}\Bigl(\dfrac{\partial L}{\partial \omega}\Bigr) -  2 \Bigl\langle \dfrac{d}{dt}\Big(\frac{\partial L}{\partial \omega}\Bigr), \dot{x} \Bigr\rangle x -  \Bigl\langle \dfrac{d}{dt}\Big(\frac{\partial L}{\partial \omega}\Bigr), 
x \Bigr\rangle \dot{x}\\
& \quad -2 \Bigl\langle \frac{\partial L}{\partial \omega}, \dot{x} \Bigr\rangle \dot{x} -\Bigl\langle \frac{\partial L}{\partial \omega}, \ddot{x} \Bigr\rangle x -\Bigl\langle \frac{\partial L}{\partial \omega}, x \Bigr\rangle \ddot{x}.
\end{flalign*} 
Replacing the above into (\ref{equacao}), taking the inner product of the resulting expression with $x$ and using also the fact that $\langle x,x\rangle=1$, 
the scalar function $2\frac{\mu}{\lambda}$ becomes
\begin{flalign*}
2\frac{\mu}{\lambda} =& -2 \Bigl\langle \dfrac{d}{dt}\Big(\frac{\partial L}{\partial \omega}\Bigr), \dot{x} \Bigr\rangle -\Bigl\langle \frac{\partial L}{\partial \omega}, \ddot{x} \Bigr\rangle- 3 \Bigl\langle \frac{\partial L}{\partial \omega}, x \Bigr\rangle \bigl\langle\ddot{x},x\bigr\rangle - \Bigl\langle \dfrac{d}{dt}\Big(\frac{\partial L}{\partial \dot{x}}\Bigr), x\Bigr\rangle  \\
& + \Bigl\langle \frac{\partial L}{\partial x}, x \Bigr\rangle+2 \frac{\partial L}{\partial z} \Bigl\langle \frac{\partial L}{\partial \omega}, \dot{x} \Bigr\rangle + \frac{\partial L}{\partial z} \Bigl\langle \frac{\partial L}{\partial \dot{x}}, x \Bigr\rangle.
\end{flalign*}
Plugging the obtained expression for $2 \frac{\mu}{\lambda}$ into (\ref{equacao}), we get the desired result. 
\end{proof}

The next result follows immediately from the above theorem.

\begin{corollary} {\rm(\cite{2016AbrMacMar})} \label{cor1}
If $L$ does not depend on $\omega$, then problem {\rm ($\P_V$)} coincides with the first--order Herglotz problem and the Euler--Lagrange equation reduces to
\begin{align*}\label{E-L-z-1}
\dfrac{D}{dt}\Bigl(\dfrac{\partial L}{\partial \dot{x}}\Bigr)-\dfrac{\partial L}{\partial x}+\Bigl\langle \dfrac{\partial L}{\partial x},x\Bigr\rangle x -\dfrac{\partial L}{\partial z} \Bigl(\dfrac{\partial L}{\partial \dot{x}} -\Bigl\langle \dfrac{\partial L}{\partial \dot{x}} ,x\Bigr\rangle x 
\Bigr) =0.
\end{align*}
\end{corollary}

From Theorem \ref{T:ELsecondOrder}, we also obtain the following result that gives the Euler--Lagrange equation for the second--order problem  of the classical calculus of variations on $S^n$ written using covariant derivatives. The result is obtained immediately from the generalized Euler--Lagrange equation (\ref{E-L(2)}) if one considers $\frac{\partial L}{\partial z}=0$. To the best of the authors knowledge, this result is not known in the literature.  

\begin{corollary} \label{cor2}
The Euler--Lagrange equation for the classical variational problem \vspace*{-.2cm}
\begin{equation*}
\min_{x\in C^4([0,T],S^n)}  ~ \dis\int_0^T L(t,x,\dot{x},\omega)\,dt
\end{equation*}
subject to the boundary conditions (\ref{bound_condition}) is 
\begin{align}
\begin{split}\label{E-L-no-z}
& \dfrac{D^2}{dt^2}\Bigl(\dfrac{\partial L}{\partial \omega}\Bigr) -\dfrac{D}{dt}\Bigl(\dfrac{\partial L}{\partial \dot{x}}\Bigr)+\dfrac{\partial L}{\partial x}-\Bigl\langle \dfrac{\partial L}{\partial x},x\Bigr\rangle x -\bigl\langle \ddot{x},x\bigr\rangle \dfrac{\partial L}{\partial \omega}-\Bigl\langle \dfrac{d}{dt}\Bigl(\dfrac{\partial L}{\partial \omega}\Bigr),x\Bigr\rangle \dot{x}\\
&- 2 \Bigl\langle \frac{\partial L}{\partial \omega}, 
x \Bigr\rangle \ddot{x}  -2 \Bigl\langle \frac{\partial L}{\partial \omega}, 
\dot{x}\Bigr\rangle \dot{x} +
3\Bigl\langle  \dfrac{\partial L}{\partial \omega},x\Big\rangle\bigl\langle\ddot{x},x\bigr\rangle x =0.
\end{split}
\end{align}
\end{corollary}
\vspace*{.2cm}

\begin{remark} {\rm 
If the geometry of the configuration space is not taken into account, 
the Euler--Lagrange equation {\rm (\ref{E-L(2)})} becomes
\begin{equation*} 
\dfrac{\partial L}{\partial x}-\dfrac{d}{dt}\Bigl(\dfrac{\partial L}{\partial \dot{x}}\Bigr)+\dfrac{d^2}{dt^2}\Bigl(\dfrac{\partial L}{\partial \ddot{x}}\Bigr)+\dfrac{\partial L}{\partial z}\dfrac{\partial L}{\partial \dot{x}}-\dfrac{\partial L}{\partial z}  \dfrac{d}{dt}\Bigl(\dfrac{\partial L}{\partial \ddot{x}} \Bigr)  -\dfrac{d}{dt}\Bigl(\dfrac{\partial L}{\partial z} \dfrac{\partial L}{\partial \ddot{x}}\Bigr)+\Bigl(\dfrac{\partial L}{\partial z}\Bigr)^2 \dfrac{\partial L}{\partial \ddot{x}}=0.
\end{equation*}
As it can be easily checked, this corresponds exactly to the Euler--Lagrange equation for the second--order Herglotz problem in Euclidean spaces studied in \rm{\cite{Simao+Delfim+Natalia2014}}. }
\end{remark}

\section{The Second--Order Herglotz Problem on $S^n$ from an Optimal Control Viewpoint} \label{OControlP}

In this section, we define an optimal control problem corresponding to the second--order generalized Herglotz problem introduced in Section \ref{2stOrderVP} and prove that its extremals, in case of existence, satisfy a certain set of Hamiltonian equations.
\begin{quote}
{\textbf{Problem ($\P_C$)}:} 
{Determine the control $u\in C^2([0,T],T^2S^n)$ that minimizes the final value of the function $z$:
\begin{equation*}
\dis{\min_{u}}~~z(T), 
\end{equation*}
where $x\in C^4\bigl([0,T],S^n\bigr)$, $v\in C^3\bigl([0,T],TS^n\bigr)$ and $z\in C^1\bigl([0,T],\cR\bigr)$ satisfy  the control system
\[
\left\{
\begin{array}{l}
\dot{x}(t)=v(t)
\\
\dot{v}(t)=u(t)
\\
\dot{z}(t)=L\Bigl(t,x(t),v(t),u(t)-\langle u(t), x(t)\rangle x(t),z(t)\Bigr),
\end{array}
\right.
\]
subject to the boundary conditions (\ref{init_condition})--(\ref{bound_condition}).
}
\end{quote}


From a geometric point of view, the state space of the control problem is $TS^n\times \mathbb{R}$ and the control bundle is the tangent bundle $T^2S^n$. 


Analogously to what has been done in the variational approach, one will look at the above optimal control problem as a pure constrained optimal control problem in the Euclidean space $\cR^{n+1}$ subject to (\ref{eq:constraint}): 

\begin{quote}
{\textbf{Problem  ($\overline{\P}_C$)}:}  
Determine the control  $u\in C^2([0,T],\cR^{n+1})$ 
that minimizes the final value of the function $z(T)$:
\begin{equation*}
\dis{\min_u} ~~z(T),
\end{equation*}
where the trajectories $x\in C^4([0,T],\cR^{n+1})$, $v\in C^3([0,T],\cR^{n+1})$  associated to $u$, and the trajectory $z\in C^1([0,T],\cR)$  satisfy the control system 
\begin{equation} \label{eq:CS}
\left\{
\begin{array}{l}
\dot{x}(t)=v(t)

\\
\dot{v}(t)=u(t)

\\
\dot{z}(t)=L\Bigl(t,x(t),v(t),u(t)-\langle u(t), x(t)\rangle x(t),z(t)\Bigr),
\end{array}
\right.
\end{equation}
subject to the boundary conditions (\ref{init_condition})--(\ref{bound_condition})
and to the constraint (\ref{eq:constraint}).
\end{quote}




In this case, the total Hamiltonian is defined by
\begin{align*} 
H(t,x,v,u,z,p_x,p_v,p_z,\lambda)&=\langle p_x, v\rangle +\langle p_v, u \rangle+p_z L(x,v,u-\langle u, x\rangle x, z)\\
&\quad +\lambda\left(\langle x, x\rangle-1\right),
\end{align*}
where $x,\,v,\,u,\,p_x,\,p_v\in\cR^{n+1}$ and $z,\,p_z,\,\lambda\in\cR$.


\begin{theorem} \label{thmOCP}
If $u$ is a solution of problem {\rm ($\overline{\P}_C$)} and $(x,v,z)$ is the associated optimal state trajectory, then there exist  a costate trajectory \linebreak $(p_x,p_v,p_z)\in C^1([0,T], \cR^{2n+1}\times\cR)$ and $\lambda\in C^1([0,T],\cR)$ such that the following  conditions hold:
\begin{itemize}
\item the Hamiltonian equations 
\begin{equation} \label{eeq2-2}
\left\{ \begin{array}{ll}
\dot{x}=v\\
\dot{v}=u \\
\dot{z}=L\bigl(t,x,v,u-\langle u, x\rangle x,z\bigr) \\
\dot{p}_x=-p_z \dfrac{\partial L}{\partial x} +p_z\Bigl\langle x, \dfrac{\partial L}{\partial \omega}\Bigr\rangle u +p_z\langle u,x\rangle \dfrac{\partial L}{\partial \omega} -2\lambda x\\
\dot{p}_v= -p_x-p_z\dfrac{\partial L}{\partial v} \\
\dot{p}_z=-p_z \dfrac{\partial L}{\partial z},
\end{array}
\right.
\end{equation}
\item the optimality condition
\begin{equation} \label{eq:Optimality}
 p_v + p_z \Bigl(\dfrac{\partial L}{\partial \omega} - \Bigl\langle \dfrac{\partial L}{\partial \omega},x\Bigr\rangle x\Bigr)=0, \\
\end{equation}
\item the transversality condition 
\begin{equation*}
p_z(T)=-1,
\end{equation*}
\end{itemize}
where $\omega=u-\langle u,x\rangle x$.

\end{theorem}
\begin{proof}
By using the dynamical constraints, the costate trajectories $p_x$, $p_v$, $p_z$ and $\lambda$, define the following functional: 
\begin{flalign*}
&\quad ~ J(u)\\
&=z(T) +\displaystyle\int^T_0\!\!\Bigl[\langle p_x, \dot{x}-v \rangle+\langle p_v, \dot{v}-u \rangle -\lambda[\langle x,x\rangle-1]\Bigr] dt \\
 & \quad +\displaystyle\int^T_0 p_z\bigl[ \dot{z}-L\bigl(t,x,v,u-\langle u, x\rangle x,z\bigr) \bigr] dt\\
&=z(T)+\displaystyle \int^T_0\Bigl(\langle p_x, \dot{x}\rangle+\langle p_v, \dot{v}\rangle+p_z \dot{z}\Bigr)\,dt\\
& \quad -\displaystyle\int^T_0\Bigl[\langle p_x,v\rangle+\langle p_v,u\rangle+p_z L\bigl(t,x,v,u-\langle u, x\rangle x,z\bigr)+\lambda(\langle x,x\rangle-1)\Bigr]\,dt\\
& =z(T)+\displaystyle\int^T_0\Bigl(\langle p_x, \dot{x}\rangle+\langle p_v, \dot{v}\rangle+p_z \dot{z}\Bigr)\,dt-\int^T_0H(t,x,v,u,z,p_x,p_v,p_z,\lambda)\,dt.
\end{flalign*}

Let us consider $u+\varepsilon \delta u$ an admissible variation of $u$ and the related variations of $x$, $v$, $z$ and $\lambda$:
$$
x+\varepsilon \delta x, \quad v+\varepsilon \delta v, \quad z+\varepsilon \delta z \quad \mbox{and} \quad \lambda+\varepsilon \delta \lambda,
$$
where  $\varepsilon$ is a real parameter and $\delta x \in C^4([0,T],\cR^{n+1})$, $\delta v \in C^3([0,T],\cR^{n+1})$, $\delta u \in C^2([0,T],\cR^{n+1})$ and $\delta z,\;\delta \lambda \in C^1([0,T],\cR)$ are such that \linebreak $\delta x(0)= \delta x(T) =\delta v(0)= \delta v(T) =\delta z(0)=0$. 
Note that, a necessary condition for $u$ to be an optimal control is that $\left.\frac{d}{d\varepsilon}J(u+\varepsilon\delta u)\right|_{\varepsilon=0}=0$, for all variations $\delta u$.

The first variation of $J$ is given by
\begin{flalign*}
&\quad \left.\frac{d}{d\varepsilon}J(u+\varepsilon\delta u)\right|_{\varepsilon=0} \\
&= \delta z(T)+ \int^T_0\Bigl(\Bigl\langle p_x,\dfrac{d(\delta x)}{dt}\Bigr\rangle+\Bigl\langle p_v,\dfrac{d(\delta v)}{dt}\Bigr\rangle+p_z \dfrac{d(\delta z)}{dt}\Bigr)\,dt
\\
& \quad-\int^T_0\Bigl(\Bigl\langle \frac{\partial H}{\partial x},\delta x\Bigr\rangle+\Bigl\langle \frac{\partial H}{\partial v},\delta v\Bigr\rangle+\Bigl\langle\frac{\partial H}{\partial u},\delta u\Bigr\rangle+\frac{\partial H}{\partial z}\delta z+\frac{\partial H}{\partial \lambda}\delta \lambda\Bigr)\,dt.
\end{flalign*}
Integrating now by parts, one gets
\begin{flalign*}
&\quad \left.\frac{d}{d\varepsilon}J(u+\varepsilon\delta u)\right|_{\varepsilon=0}\\
&=-\int^T_0\Bigl(\langle \dot{p}_x,\delta x\rangle+\langle \dot{p}_v,\delta v\rangle+\dot{p}_z \delta z\Bigr)\,dt +p_z(T)\delta z(T)+\delta z(T)\\
& \quad \displaystyle -\int^T_0\!\Bigl(\Bigl\langle \frac{\partial H}{\partial x},\delta x\Bigr\rangle\!+\!\Bigl\langle \frac{\partial H}{\partial v},\delta v\Bigr\rangle\!+\!\Bigl\langle\frac{\partial H}{\partial u},\delta u\Bigr\rangle\!+\!\frac{\partial H}{\partial z}\delta z\!+\!\frac{\partial H}{\partial \lambda}\delta \lambda\Bigr)\,dt \\
&  \displaystyle =-\int^T_0\Bigl[\Bigl\langle \frac{\partial H}{\partial x}+\dot{p}_x,\delta x\Bigr\rangle+\Bigl\langle \frac{\partial H}{\partial v}+\dot{p}_v,\delta v\Bigr\rangle+\Bigl(\frac{\partial H}{\partial z}+\dot{p}_z\Bigr)\delta z\Bigr]\,dt \\
&\displaystyle
\quad -\int^T_0\Bigl\langle\frac{\partial H}{\partial u},\delta u
+\frac{\partial H}{\partial \lambda}\delta \lambda\Bigr\rangle\,dt+\bigl[p_z(T)+1\bigr]\delta z(T).
\end{flalign*}

The first three equations of the Hamiltonian system (\ref{eeq2-2}) are the equations of the control system (\ref{eq:CS}). Choosing the costates $p_x$, $p_y$ and $p_z$ such that 
\begin{equation*} 
\left\{ 
\begin{array}{ll}
\dot{p}_x=-\dfrac{\partial H}{\partial x}\\
\dot{p}_v=-\dfrac{\partial H}{\partial v}\\
\dot{p}_z=-\dfrac{\partial H}{\partial z}\\
\end{array}
\right.
\end{equation*}
and $p_z(T)=-1$ one gets the last three equations of (\ref{eeq2-2}). Finally, the optimality condition (\ref{eq:Optimality}) follows from the arbitrariness of $\delta u$ and  the constraint condition $\frac{\partial H}{\partial \lambda}=0$. 
\end{proof}

\begin{remark}{\rm 
Notice that the Euler--Lagrange equation {\rm (\ref{E-L(2)})} can be obtained from the optimality necessary conditions given in Theorem {\rm \ref{thmOCP}}. In fact, differentiating {\rm (\ref{eq:Optimality})} with respect to $t$, one gets
\begin{align}\label{difeq2}
 \dot{p}_v + \dot{p}_z  \Bigl( \dfrac{\partial L}{\partial \omega}  - \Bigl\langle \dfrac{\partial L}{\partial \omega},x\Bigr\rangle x\Bigr) +p_z\dfrac{d}{dt}\Bigl(\dfrac{\partial L}{\partial \omega}-\Bigl\langle \dfrac{\partial L}{\partial \omega},x\Bigr\rangle x \Bigr)=0.
\end{align}
Plugging {\rm(\ref{eeq2-2})} into {\rm(\ref{difeq2})}, one obtains
\begin{align}\label{difeq3}
 p_x+p_z\left[\dfrac{\partial L}{\partial v} +\dfrac{\partial L}{\partial z} \Bigl(\dfrac{\partial L}{\partial \omega}  -\Bigl\langle \dfrac{\partial L}{\partial \omega},x\Bigr\rangle x\Bigr) -\dfrac{d}{dt}\Bigl(\dfrac{\partial L}{\partial \omega}-\Bigl\langle \dfrac{\partial L}{\partial \omega},x\Bigr\rangle x \Bigr)\right]=0,
\end{align}
and by differentiating equation {\rm(\ref{difeq3})} with respect to $t$, it yields 
\begin{align*}
& \dot{p}_x+\dot{p}_z\left[\dfrac{\partial L}{\partial v} +\dfrac{\partial L}{\partial z} \Bigl(\dfrac{\partial L}{\partial \omega}  -\Bigl\langle \dfrac{\partial L}{\partial \omega},x\Bigr\rangle x\Bigr) -\dfrac{d}{dt}\Bigl(\dfrac{\partial L}{\partial \omega}-\Bigl\langle \dfrac{\partial L}{\partial \omega},x\Bigr\rangle x \Bigr)\right]
\\
&+p_z \left[\dfrac{d}{dt}\dfrac{\partial L}{\partial v} + \dfrac{d}{dt}\Bigl(\dfrac{\partial L}{\partial z} \dfrac{\partial L}{\partial \omega}-\dfrac{\partial L}{\partial z}  \Bigl\langle \dfrac{\partial L}{\partial \omega},x\Bigr\rangle x  \Bigr)- \dfrac{d^2}{dt^2}\Bigl(\dfrac{\partial L}{\partial \omega} - \Bigl\langle \dfrac{\partial L}{\partial \omega},x\Bigl\rangle x \Bigr) \right]=0.
\end{align*}
Using again  {\rm (\ref{eeq2-2})}, one has 
\begin{align*}
\begin{split} 
& p_z \left[ \dfrac{\partial L}{\partial x}-\dfrac{d}{dt}\dfrac{\partial L}{\partial v} -\Bigl\langle x, \dfrac{\partial L}{\partial \omega}\Bigr\rangle u -\langle u,x\rangle \dfrac{\partial L}{\partial \omega} + \Bigl( \dfrac{\partial L}{\partial z}\Bigr)^2\Bigl( \dfrac{\partial L}{\partial \omega} -  \Bigl\langle \dfrac{\partial L}{\partial \omega},x\Bigr\rangle x\Bigr)\right.\\
&\left.+\dfrac{\partial L}{\partial z}\dfrac{\partial L}{\partial v}-\dfrac{\partial L}{\partial z}\dfrac{d}{dt}\dfrac{\partial L}{\partial \omega} +\dfrac{\partial L}{\partial z}  \dfrac{d}{dt}\Bigl(\Bigl\langle \dfrac{\partial L}{\partial \omega},x\Bigr\rangle x \Bigr) - \dfrac{d}{dt}\Bigl(\dfrac{\partial L}{\partial z} \dfrac{\partial L}{\partial \omega}-\dfrac{\partial L}{\partial z}  \Bigl\langle \dfrac{\partial L}{\partial \omega},x\Bigr\rangle x  \Bigr)\right.\\
& \left.+ \dfrac{d^2}{dt^2}\Bigl(\dfrac{\partial L}{\partial \omega} - \Bigl\langle \dfrac{\partial L}{\partial \omega},x\Bigl\rangle x \Bigr) \right]=-2\lambda x.
\end{split}
\end{align*}
Taking now the inner product of the above equation with $x$ obtain $\lambda$ and plug it on the above. Finally, using the fact that $v=\dot{x}$ and $u=\ddot{x}$, the Euler--Lagrange equation {\rm (\ref{E-L(2)})} comes after some straightforward computations. }
\end{remark}

\begin{remark} {\rm
Notice that an optimal control problem with pure state constraints can be addressed by many different ways. For more details on this subject we refer, for instance, {\rm\cite{Trelat2003,JacobsonLele1971}}. 
Notice also that the regularity conditions considered in Theorem {\rm\ref{thmOCP}} can be relaxed by doing some convenient adaptations.  For a matter of coherence, in the formulation of the optimal control problem, we kept the regularity conditions considered in the variational approach given in Section {\rm\ref{2stOrderVP}}. }
\end{remark}

\section{Ilustrative Examples}

In this section we present two classical examples that can be obtained directly from our main results. Namely, the cubic polynomials on $S^n$ and the dynamics of the simple pendulum moving in a resistive medium.  
\subsection{{\textbf{Cubic Polynomials on $S^n$}}} 

\label{RCP}
Cubic polynomials on Riemannian manifolds were defined in Noakes {\it et al.} \cite{NoaHeiPad:1989} as the solutions of the Euler--Lagrange equation corresponding to the variational problem:
$$
\dis\min_{x} \; ~\frac{1}{2}\int_0^T \Bigl\langle \frac{D^2 x}{d t^2}, \frac{D^2 x}{d t^2} \Bigr\rangle \;dt,
$$
where $x$ satisfies the boundary conditions $x(0)=x_0$, $x(T)=x_T$, $\dot{x}(0)=v_0$ and $\dot{x}(T)=v_T$, for some $x_0, \,x_T\in S^n$, $v_0\in T_{x_0}S^n$ and $v_{T}\in T_{x_T}S^n$. 
Therefore, the problem of finding cubics on $S^n$ can be seen  as a particular case of the Herglotz problem by considering the Lagrangian $L$ defined by
$$
L(t,x,\dot{x},\omega)=\frac{1}{2}\bigl\langle \omega,\omega\bigr\rangle.
$$ 
Notice that, for this particular Lagrangian, $
\frac{\partial L}{\partial x}=\frac{\partial L}{\partial  \dot{x}}=\frac{\partial L}{\partial  z}=0$ and $\frac{\partial L}{\partial \omega}=\omega.$
Moreover,
\begin{flalign*}
\dfrac{D^2}{dt^2}\Bigl(\dfrac{\partial L}{\partial \omega}\Bigr) &= \dfrac{D}{dt} \Bigl( \dddot{x} - \langle \dddot{x},x\rangle x-\langle \ddot{x},x\rangle\dot{x} \Bigr)\\
&= \ddddot{x} - \langle \ddddot{x},x\rangle x- 2\langle\dddot{x},x \rangle \dot{x}-\langle \ddot{x},\dot{x}\rangle\dot{x} -\langle \ddot{x},x\rangle \ddot{x}+\langle \ddot{x},x\rangle^2 x.
\end{flalign*}
Notice also that
$
\bigl\langle \frac{\partial L}{\partial \omega}, 
x \bigr\rangle =0
$
and
$\bigl\langle \frac{d}{dt}\bigl(\frac{\partial L}{\partial \omega}\bigr),x\bigr\rangle \dot{x}=-\langle \ddot{x},\dot{x}\rangle \dot{x}$. Therefore, in this particular case, the Euler--Lagrange equation (\ref{E-L-no-z}) reduces  to
\begin{flalign*}
&\ddddot{x} - \langle \ddddot{x},x\rangle x- 2\langle\dddot{x},x \rangle \dot{x}-2\langle \ddot{x},\dot{x}\rangle\dot{x} -\langle \ddot{x},x\rangle \ddot{x} - \langle \ddot{x},x\rangle \ddot{x}+2\langle\ddot{x},x\rangle^2 x=0.
\end{flalign*}
Taking now into account that $\langle x,x\rangle=1$, we can derive the following identities 
\begin{flalign*}
&\langle \dot{x},x\rangle=0,\\
&\langle \ddot{x},x\rangle=-\langle \dot{x},\dot{x}\rangle,\\
&\langle \dddot{x},x\rangle=-3\langle \ddot{x},\dot{x}\rangle,\\
&\langle \ddddot{x},x\rangle=-4\langle \dddot{x},\dot{x}\rangle-3\langle \ddot{x},\ddot{x}\rangle.
\end{flalign*}
Making use of the above, the Euler--Lagrange equation becomes
\begin{flalign*}
\ddddot{x}+4\langle \dddot{x},\dot{x}\rangle x+3\langle \ddot{x},\ddot{x}\rangle x+4\langle \ddot{x},\dot{x}\rangle \dot{x} -2 \langle \ddot{x},x\rangle \ddot{x}+2 \langle \ddot{x},x\rangle^2x=0,
\end{flalign*}
which is exactly the equation that characterizes cubic polynomials on $S^n$.

In \cite{2013AbrCamClem}, it can be found a geometric Hamiltonian formulation of the Riemannian cubic polynomials problem and some examples, including the $2$--sphere $S^2$. This paper has also interesting references in the context of cubic polynomials.


\subsection{{\textbf{Simple Pendulum}}} \label{SP}

The simple pendulum is formed of a light, stiff, inextensible rod of length $l$ with a bob of mass $m$ that is moving in the $2$--dimensional Euclidean space.

In what follows, we assume that the gravitational field is uniform, being $g$ the acceleration of gravity and consider, without loss of generality, that $l=1$.

The trajectory of the simple pendulum is therefore a curve lying in $S^1$ that we denote by $t\mapsto x(t)=\bigl(x_1(t),x_3(t)\bigr)\in S^1$. For convenience of notation, we sometimes drop the index $t$ and write simply $x$.

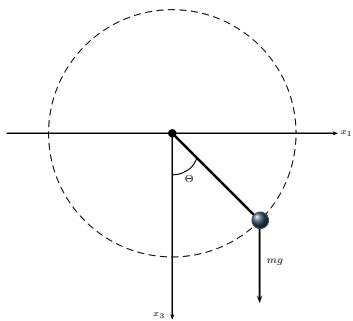
\begin{figure}[h]
\psscalebox{0.55}{
\begin{pspicture}(-10.4,-5.5)(15.7,3.4)
\psset{linewidth=0.5pt}
\psaxes[labels=none,ticks=none]{->}(0,0)(-4.0,0)(4.0,-4.5)
\uput[r](3.9,0){\tiny $x_1$}
\psline[linewidth=0.01,linestyle=solid](0,-0.05)(0,0.05)
\uput[l](0,-4.4){\tiny $x_3$}
\pscircle[linestyle=dashed](0,0){3}
\pscircle*(0,0){0.1}
\psplot[linecolor=black,linewidth=1.75pt]{0}{2.11}{0 1 x mul sub}
\psset{linewidth=0.75pt}
 \pscurve(0,-1)(0.15,-0.99)(0.45,-0.83)(0.6,-0.6)
\uput[l](0.7,-1.1){\scriptsize $\Theta$}
\uput[r](2.11,-3.11){\scriptsize $mg$}
\psset{linewidth=1.25pt}
 \psline{->}(2.11,-2.11)(2.11,-4.11)
\put(1.92,-2.31){\begin{tikzpicture}\draw[draw=fg!50!dblue!90!bg, inner sep=1pt, shading=ball,circle, ball color=dblue](1.5,-1.5)  circle (0.2cm);\end{tikzpicture}} 
\end{pspicture}
}

\caption{Illustration of a simple pendulum whose trajectory is contained in $S^1$.} \label{fig1}
\end{figure}

For such a motion, the kinetic energy is given by
$$
T=\dfrac{1}{2}m\langle \dot{x},\dot{x}\rangle,
$$
which depends only on $\dot{x}$ and the potential $U$, depending only on the position $x$ of the bob, is given by
$$
U=-mg\langle x,e_3\rangle,
$$
where $e_3$ stands for the unit vector $(0,1)$. To highlight these dependences we write $T(\dot{x})$ and $U(x)$. The Lagrangian function is, in this case, defined by
$$
L(x,\dot{x})=T(\dot{x})-U(x)=\dfrac{1}{2}m\langle \dot{x},\dot{x}\rangle + mg\langle x,e_3\rangle.
$$
Since, $\frac{\partial L}{\partial \dot{x}}=m\dot{x}$ and $\frac{\partial L}{\partial x} =mg\,e_3$, the Euler--Lagrange equation for the motion of the simple pendulum is, according to Corollary \ref{cor1}, given by
$$
\ddot{x}+\langle \dot{x},\dot{x} \rangle\, x -g\, e_3 +g \,\bigl\langle e_3,x\bigr\rangle\, x =0.
$$

If, in addition, the pendulum is moving in a resistive medium with a resistive force that is proportional to the velocity $\dot{x}$,  the Lagrangian function is defined by
$$
L(x,\dot{x},z)=T(\dot{x})-U(x)-\alpha z=\dfrac{1}{2}m\langle \dot{x},\dot{x}\rangle -U(x)-\alpha z,
$$
where $\alpha$ denotes a positive real number. From Corollary \ref{cor1}, the Euler--Lagrange equation is
$$
\dfrac{D}{dt}\dfrac{\partial L}{\partial \dot{x}} - \dfrac{\partial L}{\partial x} +\Bigl\langle \dfrac{\partial L}{\partial x} ,x \Bigr\rangle x- \dfrac{\partial L}{\partial z}\Bigl(\dfrac{\partial L}{\partial \dot{x}}-\Bigl\langle \dfrac{\partial L}{\partial \dot{x}} ,x \Bigr\rangle x\Bigr)=0,
$$  
which, in this case, reduces to
\begin{equation}\label{pendulum_eq}
\ddot{x}+\langle \dot{x},\dot{x} \rangle x -g\,e_3 + g\, \bigl\langle e_3,x\bigr\rangle\, x + \alpha \dot{x}=0,
\end{equation}
where equality $\langle \ddot{x},x \rangle=-\langle \dot{x},\dot{x} \rangle$ and the fact that $\langle \dot{x},x\rangle=0$ have been used.
To prove the equivalence between (\ref{pendulum_eq}) and the classical equation for the simple pendulum moving in a resistive medium, let us write $x_1$ and $x_3$ using the angular displacement $\Theta$ represented in Figure \ref{fig1}:
$$
\left\{\begin{array}{ll}
x_1=\sin \Theta\\
x_3=\cos \Theta.
\end{array}\right.
$$
Since $\langle \dot{x},\dot{x}\rangle=\dot{\Theta}^2$ and 
$$
\left\{\begin{array}{ll}
\ddot{x}_1=\ddot{\Theta}\cos \Theta - \dot{\Theta}^2\sin\Theta \\
\ddot{x}_3=-\ddot{\Theta}\sin \Theta - \dot{\Theta}^2\cos\Theta,\end{array}\right.
$$
equation (\ref{pendulum_eq}) is equivalent to
\begin{equation}\label{syst}
\left\{\begin{array}{ll}
\ddot{\Theta}\cos \Theta - \dot{\Theta}^2\sin\Theta + \dot{\Theta}^2\sin\Theta + g \sin \Theta \cos \Theta +\alpha \dot{\Theta}\cos\Theta=0\\
-\ddot{\Theta}\sin \Theta -\dot{\Theta}^2\cos\Theta + \dot{\Theta}^2\cos\Theta -g+g\cos^2\Theta -\alpha \dot{\Theta}\sin\Theta=0. 
\end{array}\right.
\end{equation}
Multiplying the first equation of (\ref{syst}) by $\cos\Theta$ and the second equation by $\sin\Theta$ and afterwards subtracting those equations one gets 
\begin{equation}\label{classical}
\ddot{\Theta}+g\sin \Theta +\alpha \dot{\Theta}=0.
\end{equation}
If one considers $\alpha =\dfrac{K}{m}$, equation (\ref{classical}) describes the motion of a simple pendulum moving in a resistive medium having $K$ as the coefficient resistive (see, for instance, \cite{Mickens,Nelson}).

\section{Conclusions}\label{Concluding Remarks}
The second--order Herglotz problem has been extended to the Euclidean sphere $S^n$ and first--order optimality conditions have been derived following variational and optimal control approaches. The problem differs from the corresponding higher--order problem on Euclidean spaces not only by the introduction of a holonomic constraint, but also by the use of covariant derivatives. Consequently, the nonlinearity of covariant derivatives difficults the computations and the generalized Euler--Lagrange equation turned out to be much more complex than its counterpart on the Euclidean spaces.

An important feature that deserved our attention was the fact that the Euler--Lagrange equation for the classical calculus of variations problem on $S^n$ could be obtained as a particular case of the Herglotz variational problem. Therefore, many of the problems that appear in the literature requiring classical approaches can be seen as particular cases of the higher--order Herglotz problem presented in this paper. Such is the case of the Riemannian cubic polynomials on $S^n$ illustrated in subsection 5.1, or the elastic curves treated, for instance, in \cite{Cam:00}. The Herglotz problem on $S^n$, when the Lagrangian $L$ depends explicitly on the variable $z$, finds also applicability on the equation of motion of a simple pendulum in a resistive medium as shown in subsection 5.2.

One of our future aims is the study of Herglotz problems for other Riemannian manifolds, like Lie groups or symmetric spaces, that play important roles in applications arising in physics and engineering.

Other challenging question is to generalize the results presented in this paper to fractional Herglotz problems with delayed arguments using the rich framework of fractional optimal control theory \cite{Jarad:10,Jarad:12,Bahaa}.


\section*{Acknowledgements}
The work of L\'{i}gia Abrunheiro and Nat\'{a}lia Martins was supported by Portuguese funds through the \emph{Center for Research and Development in Mathematics and Applications} (CIDMA)  and the \emph{Portuguese Foundation for Science and Technology} (``FCT--Funda\c{c}\~ao para a Ci\^encia e a Tecnologia''), within project UID/MAT/04106/2013.
Lu\'{i}s Machado acknowledges ``Funda\c{c}\~ao para a Ci\^encia e a Tecnologia" (FCT--Portugal) and COMPETE 2020 Program for financial support through project UID-EEA-00048-2013.
 
The authors would like to thank the reviewers for their valuable suggestions to improve the quality of the paper.


{}

\end{document}